\documentclass[12pt]{amsart}
\usepackage[utf8]{inputenc}
\usepackage{amsmath,amssymb,amsthm}
\usepackage{url}
\usepackage{hyperref}
\usepackage{tikz-cd}
\usepackage{stmaryrd}

\usepackage{tikz} 
\usetikzlibrary{external} 
\usepackage{adjustbox}
\definecolor{color0}{HTML}{000000}
\definecolor{color1}{HTML}{000000}
\definecolor{bgColor}{HTML}{ffffff}
\usetikzlibrary {arrows.meta,bending} 

\theoremstyle{definition}
\newtheorem{definition}{Definition}[section]

\newtheorem{question}[definition]{Question}

\newtheorem{claim}[definition]{Claim}
\newtheorem{subclaim}[definition]{Subclaim}

\theoremstyle{plain}
\newtheorem*{theorem*}{Theorem}

\theoremstyle{plain}
\newtheorem{theorem}[definition]{Theorem}
\newtheorem{proposition}[definition]{Proposition}

\newtheorem{corollary}[definition]{Corollary}

\DeclareMathOperator{\cf}{cf}
\DeclareMathOperator{\cof}{cof}

\DeclareMathOperator{\otp}{otp}
\DeclareMathOperator{\dom}{dom}

\DeclareMathOperator{\acc}{acc}

\DeclareMathOperator{\Ord}{Ord}
\DeclareMathOperator{\osc}{osc}

\usepackage{amsmath,amssymb,amsthm, xcolor, enumitem, comment, todonotes}

\title{Strong colorings based on oscillations}
\author{Stevo Todorcevic and Jing Zhang}
\date{}

\address[Todorcevic]{\hfill\break Department of Mathematics, University of Toronto,
Canada
\hfill\break stevo@math.toronto.edu}

\address{\hfill\break Institut de Math\'ematiques de Jussieu, CNRS, Paris,
France
\hfill\break stevo.todorcevic@imj-prg.fr}

\address{\hfill\break Matemati\v{c}ki Institut, SANU, Belgrade, Serbia
\hfill\break stevo.todorcevic@sanu.ac.rs}

\address[Zhang]{\hfill\break Department of Mathematics\\
University of North Texas \\
USA
\hfill\break Jing.Zhang2@unt.edu}

\thanks{
{\it 2010 \emph{MSC}.} 03E02, 03E10 \newline
{The research on this paper is partially supported by
grants from NSERC(455916),  CNRS(UMR7586) and SFRS(7750027-SMART)
}}

\begin{document}
\maketitle
\begin{abstract}
We show that for any uncountable cardinal $\kappa$, there is a coloring $c: [\kappa]^2\to \omega$ such that $c''A \otimes B = \omega$ for any $A, B\subseteq \kappa$ of order type $\omega_1$ that are stationary in their common supremum. In particular, the stationary version of Erd\H{o}s-Rado theorem and the higher dimensional Friedman's property are both inconsistent. We demonstrate that the theorem is optimal in various ways.
\end{abstract}

\section{Introduction}
Ever since the work by Sierpinski \cite{Sierpinski}, it has been known that Ramsey theorem fails for the first uncountable cardinal $\omega_1$, namely, there is $c: [\omega_1]^2\to 2$ without uncountable homogeneous subset. Later, works by Laver \cite{LaverPartition}, Galvin and Shelah \cite{GalvinShelah}, increase the number of colors to 3 and 4 respectively. These methods are based on the existence of certain uncountable linear orderings, which makes generating more colors difficult. In \cite{TodorcevicFirstOscillation}, Todorcevic came up with the method of oscillation and was able to produce a coloring $c: [\omega_1]^2\to \omega$ such that for any \emph{stationary} $A, B\subseteq \omega_1$, it is the case that $c'' A\otimes B = \omega$. Here $A\otimes B =_{\mathrm{def}}\{(\alpha,\beta): \alpha<\beta, \alpha\in A, \beta\in B\}$. In \cite{TodorcevicWalks87}, the method of minimal walks on ordinals was introduced, which was combined with the previous method of oscillation to produce a coloring $c: [\omega_1]^2\to \omega$ such that for any \emph{uncountable} $A\subseteq \omega_1$, $c''[A]^2 =\omega$. Finally, the rectangular version was proved by Moore \cite{MooreLspace} (see also \cite{RinotTodorcevic} for generalizations to other cardinals), by oscillating lower traces that result from walks on ordinals.

In this note, we apply the method of oscillation to prove the following theorem.

\begin{theorem}\label{theorem: main}
For every uncountable cardinal $\kappa$, there exists a coloring $c: [\kappa]^2\to \omega$ such that  $c'' A \otimes B = \omega$ for any $A, B\subseteq \kappa$ of order type $\omega_1$ that are stationary in $\sup A = \sup B$.
\end{theorem}

Recall that an instance of Erd\H{o}s-Rado theorem  \cite{ErdosRado} gives that if $\kappa=(2^\omega)^+$, then for any $c: [\kappa]^2\to \omega$, there exists $A\subseteq \kappa$ of order type $\omega_1$ such that $c\restriction [A]^2$ is constant. Theorem \ref{theorem: main} shows that it is inconsistent to strengthen the conclusion requiring $A$ to be stationary in its supremum.

Recall that (a version of) \emph{Friedman's property} states: for any regular $\kappa\geq \omega_2$, for any $c: \kappa \to \omega$, there exists a closed $d$ of order type $\omega_1$ such that $c\restriction d$ is constant. Martin's Maximum \cite{MM} implies that Friedman's property holds. Theorem \ref{theorem: main} implies that the 2-dimensional Friedman's property, namely when the given coloring is $c: [\kappa]^2\to \omega$, fails in a strong way.

Let us briefly discuss how the proof compares with those of the previous theorems. In general, using the method of elementary submodels, we are able to control the pattern of oscillation on final segments of two given sets. This leaves us with a task of removing initial segments whose oscillating behavior cannot be controlled in general. When $\kappa$ is $\omega_1$, such a problem can be easily solved due to the fact that any countable elementary submodel contains an initial segment of $\omega_1$. When $\kappa>\omega_1$, a way to remove initial segments was invented in \cite{Todorcevic3D}, however, at the cost of using a third coordinate.   The main innovation in this note is to come up with a binary function, $\Delta'$ (see Definition \ref{def: deltaprime}), that serves to remove undesirable initial segments in an economic way.

We finish the introduction with some definitions and notations that will be used in the proof. Let $\kappa$ and $\theta$ be two fixed ordinals for the rest of the section.

\begin{definition}
A \emph{$C$-sequence} is of the form $\langle C_\alpha: \alpha\in \Gamma\rangle$ where $\Gamma\subseteq \kappa$ and each $C_\alpha$ is a closed unbounded subset of $\alpha$.
\end{definition}

Fix a $C$-sequence $\vec{C}=\langle C_\alpha: \alpha\in \Gamma\rangle$ where $\Gamma\subseteq \kappa$.
 
\begin{definition}[Definition 8.1.4, \cite{walksbook}]\label{def: stationaryCseq}
We say $\vec{C}$ is \emph{stationary} at $\theta$ if $$\bigcup_{\alpha\in \Gamma} \acc(C_\alpha) \cap \theta$$ is a stationary subset of $\theta$.
\end{definition}

\begin{definition}[Implicitly Chapter 10.3, \cite{walksbook}]\label{def: trivialCseq}
We say $\vec{C}$ is \emph{trivial} at $\theta$ if there exists a club $D\subseteq \theta$ such that for any $\xi<\theta$, there exists $\alpha\in \Gamma$ such that $D\cap \xi \subseteq C_\alpha$.
\end{definition}

\begin{definition}[Chapter 8.1, \cite{walksbook}]
Let $\osc$ be a class function whose domain is the class $\{(D,C): D,C\subseteq \Ord\}$ such that given $(D,C)\in \dom(\osc)$, $\osc(D,C)= |D- (\sup D\cap C+1) / \sim|$ where for any $a,b\in D- (\sup D\cap C+1)$, $a\sim b$ iff there is no $c\in C$ such that $a\leq c \leq b$.

\end{definition}

\section{Proof of the main theorem}
Fix a cardinal $\kappa\geq \omega_1$.
	\subsection{Uncountable cofinality}
	We first deal with the case that $\kappa$ has cofinality $\geq \omega_1$.

	\begin{definition}
We say a $C$-sequence $\langle C_\alpha: \alpha\in \Gamma\rangle$ is \emph{stationary} and \emph{strongly non-trivial} at $\theta$ witnessed by $\xi<\cf(\theta)$ if there are stationarily many $\gamma \in \theta$ such that there exists $\alpha\in \Gamma$ with $\gamma\in acc(C_\alpha)\cap \theta$ and $\otp(C_\alpha\cap \gamma)=\xi$.
\end{definition}

	\begin{definition}\label{def: deltaprime}
For $\bar{\beta}\leq \beta$ and two closed unbounded subsets $C\subseteq \bar{\beta}, D\subseteq \beta$, define $\Delta'(D,C)=\max\{\xi \leq   \max (\bar{\beta} \cap D): \otp(C\cap \xi) = \otp(D\cap \xi)\}$. 
\end{definition}

The reason why we consider such a variant is that it ``saves" us one coordinate (compared to the oscillation function defined in \cite{Todorcevic3D}) when we are trying to generate oscillation patterns.

\begin{proposition}\label{prop: deltaprime}
Suppose that $\Gamma, \Sigma\subseteq \kappa$ with $\sup \Gamma=\sup \Sigma=\lambda\leq \kappa$ and $\otp(\Gamma)=\otp(\Sigma)=\cf(\lambda)$ such that
	\begin{itemize}
	\item  $\cf(\lambda)>\omega$ and
	\item  $\langle C_\alpha: \alpha\in \Gamma\rangle$ and $\langle C_\alpha: \alpha\in \Sigma \rangle$ are $C$-sequences that are stationary and strongly non-trivial at $\lambda$ witnessed by $\xi<\cf(\lambda)$.
\end{itemize} Then for any $n\in \omega$, there are $\{\bar{\beta},\beta\}_<\in \Gamma \otimes \Sigma$ such that $$\osc(C_\beta-\Delta'(C_\beta, C_{\bar{\beta}}), C_{\bar{\beta}}-\Delta'(C_\beta, C_{\bar{\beta}}))=n+1.$$
\end{proposition}

\begin{proof}
Fix $n\in \omega$ and a large enough regular cardinal $\theta$.
Let $\nu=\cf(\lambda)$.
Let $\vec{M}=\langle M_\alpha: \alpha<\nu\rangle$ be an internally approachable sequence of elementary submodels of $H(\theta)$ of size $<\nu$ containing relevant objects with $M_\alpha\cap \nu \in \nu$ for each $\alpha<\nu$. Let $E=\{\sup M_\alpha\cap \lambda: \alpha<\nu\}$. Then $E$ is a club in $\lambda$. 

Since $\langle C_\alpha: \alpha\in \Gamma\rangle$ and $\langle C_\alpha: \alpha\in \Sigma\rangle$ are stationary and strongly non-trivial at $\lambda$ witnessed by $\xi$, we can find elementary submodels $N_0\in N_1\prec H(({2^\theta})^+)$ containing relevant objects including $\vec{M}, \lambda$, $\xi$ and $\vec{C}$ such that for any $i<2$,
	\begin{itemize}
	\item $|N_i|<\nu$,
	\item $N_i\cap \nu \in \nu$,
	\end{itemize}
as well as 
	\begin{itemize}
	\item $\lambda_{N_0}=_{\mathrm{def}}\sup N_0\cap \lambda\in \acc(C_{\bar{\beta}'})\cap \lambda$ and $\otp(C_{\bar{\beta}'} \cap \lambda_{N_0})=\xi\in N_0\cap \nu$ for some $\bar{\beta}'\in \Gamma$,
	\item $\lambda_{N_1}=_{\mathrm{def}}\sup N_1\cap \lambda\in \acc(C_{\beta})\cap \lambda$ and $\otp(C_\gamma \cap \lambda_{N_1})=\xi\in N_1\cap \nu$ for some $\beta\in \Sigma$,
	\end{itemize}
	 As $E\in N_i$, $\otp(E\cap N_i)=N_i\cap \nu$ and $\otp((C_\beta\cup C_{\bar{\beta}'})\cap N_0)\leq \xi^2 < N_0\cap \nu$, we know that for any $\eta<\lambda_{N_0}$, $E\cap N_0 - \eta \not \subseteq (C_\beta\cup C_{\bar{\beta}'}) \cap \lambda_{N_0}$.
By the order type consideration, there exist increasing sequences $\langle \lambda_i^0\in \lambda_{N_0} : i\in \omega\rangle$,  $\langle \lambda_i^1\in \lambda_{N_1} : i\in \omega\rangle$, $\langle \xi_i^0\in N_0\cap \nu: i\in \omega\rangle$ and $\langle \xi_i^1\in N_1\cap \nu: i\in \omega\rangle$ such that for all $i\in \omega$,
	\begin{itemize} 
	\item $\langle \lambda_i^0\in \lambda_{N_0} : i\in \omega\rangle$ is cofinal in $\lambda_{N_0}$ and $\langle \lambda_i^1\in \lambda_{N_1} : i\in \omega\rangle$ is cofinal in $\lambda_{N_1}$,
	\item $\lambda_0=_{\mathrm{def}}\lambda^0_0=\lambda^1_0\in E\cap N_0 - (C_\beta\cup C_{\bar{\beta}'})$,
	\item  $\lambda_i^0\in E\cap N_0-C_{\bar{\beta}'}$,
	\item  $\lambda_i^1\in E\cap N_1-C_{\beta}$,
	\item $(\lambda_i^0, \lambda_{i+1}^0)\cap C_{\bar{\beta}'}\neq \emptyset$,
	\item $(\lambda_i^1, \lambda_{i+1}^1)\cap C_{\beta}\neq \emptyset$,
	\item $\otp(C_{\bar{\beta}'}\cap \lambda_0^0)=\xi_0^0$,
	\item $\otp(C_{\beta}\cap \lambda_0^1)=\xi_0^1$,
	\item $\otp((\lambda_i^0, \lambda_{i+1}^0)\cap C_{\bar{\beta}'})=\xi_{i+1}^0$, and
	\item $\otp((\lambda_i^1, \lambda_{i+1}^1)\cap C_{\beta})=\xi_{i+1}^1$.
	\end{itemize}
We may assume that $\xi^0_0>\xi^1_0$ since $\otp(C_{\beta}\cap \lambda_{N_0})<\otp(C_\beta\cap \lambda_{N_1})=\xi$. As $\otp(C_{\bar{\beta}'}\cap \lambda_{N_0})=\xi$, by picking $\lambda_0$ sufficiently large below $\lambda_{N_0}$, we can arrange that $\xi^0_0>\xi^1_0$ holds.

For simplicity, let us denote the models witnessing $\lambda_i^0\in E$ as $M'_i$ for $i\in \omega$. We may also assume that $\otp(C_{\bar{\beta}'} \cap \lambda_{N_0})=\xi\in M_0'\cap \nu$. Let $\bar{\lambda}\in M_{0}' \cap \lambda > \sup C_\beta\cap \lambda_0$.
Let $\exists^Q I$ abbreviate: for any $\xi<\lambda$, there exists a closed interval $I \subseteq \lambda$ with $\min I > \xi$. The standard reflection argument using the elementarity of $M_0',\cdots, M'_{n},N$ (see \cite[Lemma 8.1.2]{walksbook} for example) gives that $\exists^Q I_1 \exists^Q I_2 \cdots  \exists^Q I_n \psi(I_1,\cdots, I_n)$ where $\psi(I_1,\cdots, I_n)$ asserts there is some $\bar{\beta}\in \Gamma$ satisfying the following
	\begin{itemize}
	\item $\otp(C_{\bar{\beta}}\cap \bar{\lambda}) =\xi_0^0$,
	\item $\otp(C_{\bar{\beta}}\cap I_i)=\xi_i^0$ for $i=1,2, \cdots, n-1$,
	\item $C_{\bar{\beta}}\subseteq \bar{\lambda} \cup \bigcup_{i=1}^{n} I_i$.
	\end{itemize}
Observe that the parameters involved in the sentence $$\exists^Q I_1 \exists^Q I_2 \cdots  \exists^Q I_n \psi(I_1,\cdots, I_n)$$ are $\Gamma$, $\vec{C}$, $\bar{\lambda}$, $\{\xi_i^0: 1\leq i \leq n-1\}$, which are all in $M_0'$.
	
Let us briefly sketch how the reflection works. For each $1\leq i \leq n$, find some $\bar{\lambda}_i\in M'_i\cap \lambda$ such that $\max C_{\bar{\beta}'}\cap \lambda_i<\bar{\lambda}_i$. Consider the following interval $K_i=[\lambda_{i-1}, \bar{\lambda}_i]$ for $1\leq i < n$ and $K_n=[\lambda_{n-1},\lambda)$. Note that $K_i\in M'_i$ for each $1\leq i\leq n$. Notice that $\psi(K_1,K_2, \cdots, K_n)$ is true witnessed by $\bar{\beta}' \in \Gamma$. By the elementarity of $M'_{n-1}$, we can conclude that $\exists^Q I_n \psi(K_1,\cdots, I_n)$. Repeating the same argument for the sentence $\exists^Q I_n \psi(K_1,\cdots, I_n)$ and $M'_{n-2}$, we get $\exists^Q I_{n-1}\exists^Q I_n \psi(K_1,\cdots, I_{n-1},I_n)$. After repeating this argument $n$ times, we arrive at the conclusion that $\exists^Q I_1 \exists^Q I_2 \cdots  \exists^Q I_n \psi(I_1,\cdots, I_n)$ holds.

Find $k\in \omega$ such that $\otp(C_\beta\cap \lambda^1_k)> \Sigma_{0\leq i\leq n-1} \xi^0_{i}+1$.
Let $J_1 = (\lambda_0, \max C_{\beta}\cap \lambda_k^1+1)$, $J_2 = (\lambda_k^1, \max C_{\beta}\cap \lambda_{k+1}^1+1)$, $\cdots$, \linebreak $J_n = (\lambda_{k+n-2}^1, \max C_{\beta}\cap \lambda_{k+n-1}^1+1)$. Let us denote the models witnessing $\lambda_i^1\in E$ as $\tilde{M}_i$ for $i\in \omega$. Reflecting the statement $$\exists^Q I_1 \exists^Q I_2 \cdots  \exists^Q I_n \psi(I_1,\cdots, I_n)$$ in models $\tilde{M}_k, \cdots, \tilde{M}_{k+n-1}$, we get in $N_1$ intervals $I_1<\cdots<I_n$ and the witnessing $\bar{\beta}\in \Gamma\cap \tilde{M}_{k+n-1}$ satisfying the following:  
\begin{enumerate}
\item $\psi(I_1,\cdots, I_n)$ holds witnessed by $\bar{\beta}$, 
\item $J_1<I_1<\lambda_k^1 <J_2 < I_2<\lambda_{k+1}^1< J_3 < I_3 <\lambda_{k+2}^1 <\cdots < J_n<I_n< \lambda_{k+n-1}^1$.
\end{enumerate}

From here, it is clear that $\osc(C_{\beta}-\lambda_0, C_{\bar{\beta}}-\lambda_0)=n+1$. 

	\begin{claim}\label{claim: computedeltaprime}
	$\Delta'(C_{\beta}, C_{\bar{\beta}})\in J_1-\{\max J_1\}$.
	\end{claim}
	
	\begin{proof}[Proof of the claim]
	Let $\tau = \max \bar{\beta}\cap C_\beta$. Since $\bar{\beta}\in \lambda\cap \tilde{M}_{k+n-1}-\max J_n$, we have that $\tau = \max J_n$.
	Note first that $\Delta'(C_{\beta}, C_{\bar{\beta}})\geq \lambda_0$ since
		\begin{enumerate}
		\item $\otp(C_\beta\cap \lambda_0) = \xi^1_0 < \xi^0_0 =\otp(C_{\bar{\beta}}\cap \lambda_0)$ and 
		\item $\otp(C_\beta\cap \lambda_k^1) = \otp(C_\beta\cap (\max J_1+1))> \xi^0_0 + \xi^0_1 +1 = \otp(C_{\bar{\beta}}\cap \lambda_k^1)$.
		\end{enumerate}
As a result, there exists $\chi\in C_\beta\cap J_1-\{\max J_1\}$ such that $\otp(C_{\beta}\cap \chi)=\otp(C_{\bar{\beta}}\cap \chi)$. The following subclaim finishes the proof.

	\begin{subclaim}
	$\Delta'(C_{\beta},C_{\bar{\beta}})=\chi$.
	\end{subclaim}
It suffices to show that for any $\chi'\in (\chi, \tau]$, it is the case that $\otp(C_\beta\cap \chi')>\otp(C_{\bar{\beta}}\cap \chi')$. If $\chi'\leq \max J_1$, then $$otp(C_{\bar{\beta}}\cap \chi')=\otp(C_{\bar{\beta}}\cap \lambda_0)=\xi^0_0 = \otp(C_{\beta}\cap \chi)<\otp(C_{\beta}\cap \chi').$$ If $\chi'>\max J_1$, then $\otp(C_\beta\cap \chi')>\Sigma_{0\leq i\leq n-1} \xi^0_{i}+1 > \otp(C_{\bar{\beta}}\cap \tau)\geq \otp(C_{\bar{\beta}}\cap \chi')$. This finishes the proof of the subclaim and the claim.
	\end{proof}
To finish the proof, notice that $C_{\bar{\beta}}-\chi = C_{\bar{\beta}}-\lambda_0$ and $J_1\cap (C_{\beta}-\chi)\neq \emptyset$. As a result of the claim above, we have $$\osc(C_{\beta}-\Delta'(C_{\bar{\beta}}, C_{\beta}),C_{\bar{\beta}}-\Delta'(C_{\bar{\beta}}, C_{\beta}))=n+1.$$
\end{proof}

Fix a $C$-sequence $\langle C_\alpha: \alpha\in \kappa\rangle$ such that $\otp(C_\alpha)=\cf(\alpha)$ for any $\alpha<\kappa$. Define $c: [\kappa]^2\to \omega$ such that for any $\{\alpha,\beta\}\in [\kappa]^2$, $$c(\alpha,\beta)= \osc(C_\beta-\Delta'(C_\beta,C_\alpha), C_\alpha-\Delta'(C_\beta,C_\alpha))-1.$$
For any $\Gamma, \Sigma\subseteq \kappa$ of order type $\omega_1$ that are stationary in $\sup \Gamma =\sup \Sigma$, it is easy to see that $\Gamma$ and $\Sigma$ are both stationary and strongly non-trivial at $\sup \Gamma$ witnessed by $\omega$. By Proposition \ref{prop: deltaprime}, for any $n\in \omega$, there are $\alpha<\beta\in \Gamma\otimes \Sigma$ such that $c(\alpha,\beta)=n$, as desired.

\begin{figure}[h]
    \centering
    \hspace{-0.7cm}
   \includegraphics[scale=0.75]{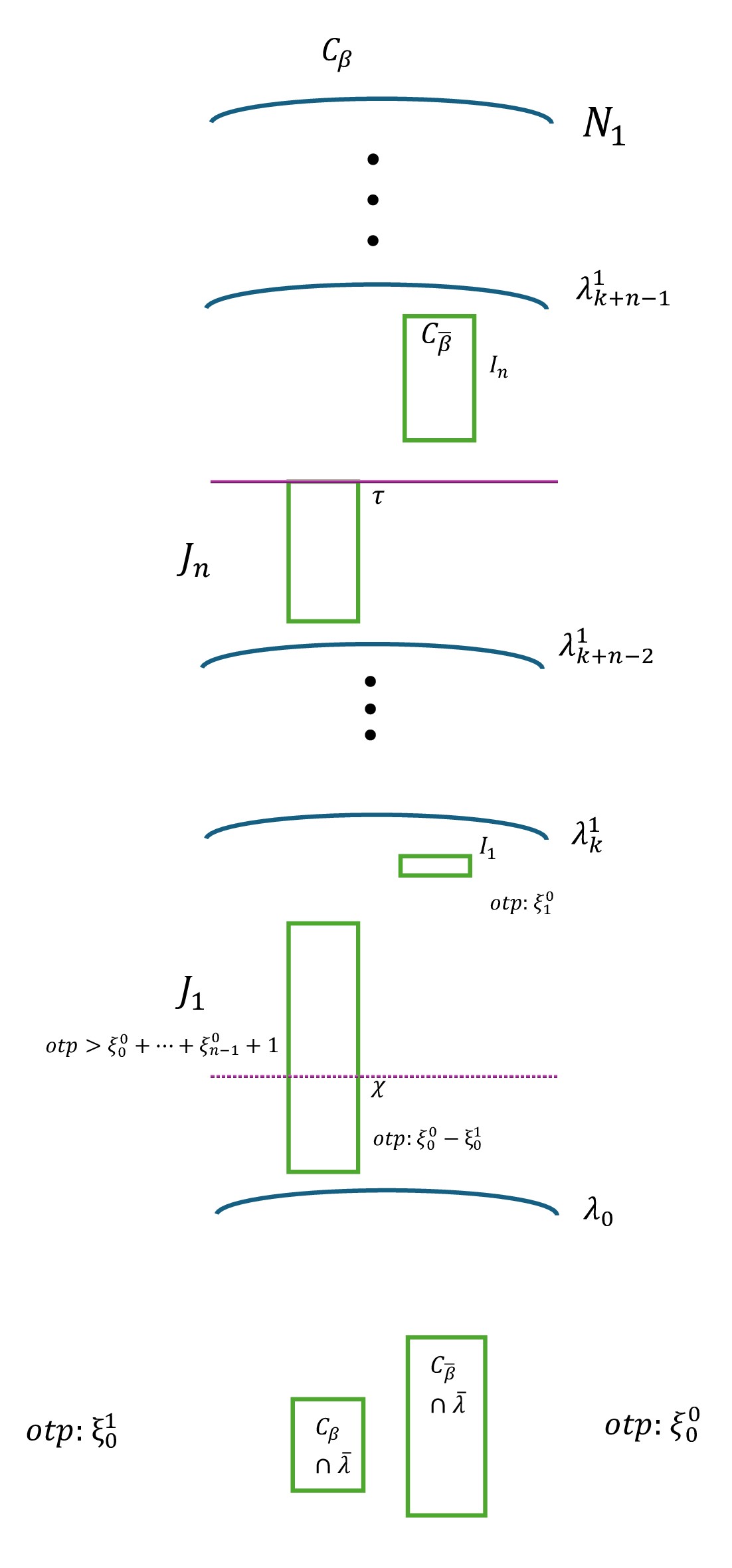}
    \caption{Illustration of the oscillation pattern in the proof of Proposition \ref{prop: deltaprime}.}
    \label{fig:osc}
\end{figure}
	
	\subsection{Countable cofinality}
Suppose $\kappa$ is a singular cardinal of countable cofinality. Fix increasing uncountable regular cardinals $\langle \kappa_i: i\in \omega\rangle$ cofinal in $\kappa$. For each $i\in \omega$, fix $c_i: [\kappa_i-\kappa_{i-1}]^2\to \omega$ witnessing that $\kappa_i\not\to [\mathrm{stat}(\omega_1)]^2_\omega$ (the convention here is that $\kappa_{-1}=0$). Define $c: [\kappa]^2\to \omega$ such that for $(\alpha,\beta)\in [\kappa]^2$, if there exists $i\in \omega$ such that $\alpha,\beta\in \kappa_{i}-\kappa_{i-1}$, then $c(\alpha,\beta)= c_i(\alpha,\beta)$. Otherwise, define $c(\alpha,\beta)=0$. Suppose we are given $A\subseteq \kappa$ such that $A$ is of order type $\omega_1$ that is stationary in its supremum. By the $<\omega_1$-completeness of the non-stationary ideal on $\omega_1$, we can find some $i\in \omega$ and a stationary $A'\subseteq A$ in $\sup A$ such that $A'\subseteq \kappa_i-\kappa_{i-1}$. By the property of $c_i$ and the definition of $c$, we have that $c''[A']^2 = c_i''[A']^2 =\omega$.

\section{Optimality}

In this section, we demonstrate that our theorem is optimal in several ways. 

First we demonstrate that the stationarity requirement is necessary. 

\begin{proposition}\label{prop: stationarity}
If $\kappa=(2^\omega)^+,$ then for any  $c: [\kappa]^2\to \omega$, there are $\Gamma,\Sigma\subseteq \kappa$ of order type $\omega_1$ with $\Gamma$ being stationary in $\sup \Gamma = \sup \Sigma$ such that $c'' \Gamma \otimes \Sigma$ is constant.
\end{proposition}

\begin{proof}
Given a coloring $c: [\kappa]^2\to \omega$, fix an elementary submodel $M\prec H(\theta)$ where
	\begin{enumerate}
	\item $\theta$ is a large enough regular cardinal,
	\item $M$ is countably closed and of size $2^\omega$, 
	\item $c\in M,$ and
	\item $\delta=M\cap \kappa\in \kappa\in \cof(\omega_1).$
	\end{enumerate}
 We can find a stationary $A\subseteq \delta$ of order type $\omega_1$ and $i\in \omega$ such that $c'' A\times \{\delta\}=\{i\}$. Since $M$ is countably closed, it is easy to recursively construct a set $\Sigma'\subseteq \delta$ cofinal in $\delta$ of order type $\omega_1$ such that for each $\alpha\in A$, for all but countably many $\eta\in \Sigma'$, it is the case that $c(\alpha,\eta)=i$. Consider the following function $f: A\to \Sigma'$ such that $f(\alpha)$ is the least $\xi \geq \sup_{\beta\in \alpha \cap A}f(\beta)+1$ such that for all $k\geq \xi$ in $\Sigma'$, it is the case that $c(\alpha,k)=i$. By the stationarity of $A$, we can find a stationary subset $\Gamma\subseteq A$ that is closed under the following function: $\alpha\in A\mapsto \min \Sigma'-(f(\alpha)+1)$, in the sense that for any $\alpha\in \Gamma$, for any $\gamma\in \alpha\cap \Gamma$, it is the case that $\min\Sigma'-(f(\gamma)+1)<\alpha$. Let $\Sigma=\{\min \Sigma' - (f(\alpha)+1): \alpha\in \Gamma\}$. Let us check that $c'' \Gamma\otimes \Sigma = \{i\}$. Given $\alpha<\beta\in \Gamma\otimes \Sigma$, by the definition of $\Sigma$, let $\alpha'\in \Gamma$ be such that $\beta=\min \Sigma' - (f(\alpha')+1)$. It suffices to show that $\alpha\leq \alpha'$. Suppose for the sake of contradiction that $\alpha'<\alpha$, by the definition of $\Gamma$, we have that $\beta=\min\Sigma'-(f(\alpha')+1)<\alpha$, contradicting with the fact that $\alpha<\beta$.\end{proof}

Next we show that the requirement that the two sets having the same supremum is also crucial.

\begin{proposition}\label{prop: MM}
Martin's Maximum implies that for any coloring  $c: [\omega_3]^2\to \omega_1$, there exist $\Gamma, \Sigma\subseteq \omega_3$ of order type $\omega_1$ that are closed in their respective suprema such that $c'' \Gamma \otimes \Sigma$ is constant.
\end{proposition}

\begin{proof}
We will only use the following two consequences of Martin's Maximum \cite{MM}: 
	\begin{enumerate}
	\item for any regular $\kappa\geq \omega_2$, any stationary $S\subseteq \kappa \cap \cof(\omega)$ contains a \emph{closed} subset of order type $\omega_1$.
	\item $2^{\omega_1}=\omega_2$.
	\end{enumerate}
Let $S^2_0$ be $\omega_2 \cap \cof(\omega)$ and $S^3_0$ be $\omega_3\cap \cof(\omega)$. For each $\beta\in S^3_0 - \omega_2$, find $i_\beta\in \omega_1$ such that $A_\beta=\{\gamma\in S^2_0: c(\gamma, \beta)=i_\beta\}$ is stationary. By Martin's Maximum, find a closed $c_\beta \subseteq A_\beta$ of order type $\omega_1$. Since $\omega_2^{\omega_1}= \omega_2$, we can find a closed subset $\Gamma\subseteq \omega_2$ of order type $\omega_1$, $i\in \omega_1$ and a stationary $S\subseteq S^3_0$ such that for any $\beta\in S$, $\Gamma=c_\beta$ and $i=i_\beta$. Apply Martin's Maximum again to find a closed $\Sigma\subseteq S$ of order type $\omega_1$. It is now easy to check that $c''\Gamma \otimes \Sigma =\{i\}$.
\end{proof}

Finally, we show that the number of colors is optimal in terms of results provable in $\mathrm{ZFC}$.

\begin{proposition}\label{prop: statChang}
It is consistent relative to the existence of suitable large cardinals that for any $c: [\omega_2]^2\to \omega_1$, there exists $A\subseteq \omega_2$ of order type $\omega_1$ that is stationary in $\sup A$ such that $|c''[A]^2|=\aleph_0$. 
\end{proposition}

\begin{proof}
This follows from a classical construction due to Kunen \cite{Kunen}. See \cite[Lemma 2.20]{GartiZhang} for a proof.
\end{proof}

The conclusion of Proposition \ref{prop: statChang} has large cardinal strength since by \cite{TodorcevicFirstOscillation}, $\square_{\omega_1}$ implies there exists a coloring $c: [\omega_2]^2\to \omega_1$, for any $A,B\subseteq \omega_2$ of order type $\omega_1$ that are stationary in $\sup A = \sup B$, it is the case that $c'' A \otimes B = \omega_1$.

\section{A consequence regarding set-systems}

Recall the following definitions from \cite{ErdosGalvinHajnal}.

\begin{definition}\label{defintion: set-systems}
$\mathcal{I}$ is a \emph{set-system with set of vertices $V$} if for any $A\in \mathcal{I}$, $A\subseteq V$ and $|A|\geq 2$.
\end{definition}

\begin{definition}\label{definition: chromatic}
Let $\mathcal{I}$ be a set-system with set of vertices $V$. The \emph{chromatic number} of $\mathcal{I}$ is the smallest cardinal $\kappa$ for which there is a function $f: V\to \kappa$ such that for any $\xi<\kappa$ and $A\in \mathcal{I}$, $A\not\subseteq f^{-1}(\{\xi\})$.
\end{definition}

\begin{definition}\label{definition: simchromatic}
Let $\{\mathcal{I}_{\nu}: \nu<\lambda\}$ be a collection of set-systems with the same set of vertices $V$. The system is said to have \emph{simultaneous chromatic number} $\kappa$ if $\kappa$ is the smallest cardinal such that there is a function $f: V\to \kappa$ such that for any $\xi<\kappa$, there is $\nu<\lambda$ such that for any $A\in \mathcal{I}_\nu$, $A\not\subseteq f^{-1}(\{\xi\})$.
\end{definition}

\begin{definition}
Let $\mathcal{I}$ be a set-system with set of vertices $V$. $P(\mathcal{I},\lambda,\kappa)$ is said to hold if there is a partition of $\mathcal{I}$ into a union of $\lambda$ disjoint set-systems $\mathcal{I}_{\nu}$ ($\nu<\lambda$) in such a way that this system has simultaneous chromatic number $\geq \kappa$.
\end{definition}

Erd\H{o}s, Galvin and Hajnal \cite{ErdosGalvinHajnal} showed that $\omega_1\not\to [stat(\omega_1)]^2_{\omega}$ implies that $P([\omega_1]^2, \aleph_0, \aleph_1)$ holds. Theorem \ref{theorem: main} gives rise to the validity of the following variant.

\begin{definition}
Let $\mathcal{I}$ be a set-system with set of vertices $V$ and let $V'\subseteq \mathcal{P}(V)$. $P(\mathcal{I},\lambda,\kappa, V')$ is said to hold if there is a partition of $\mathcal{I}$ into a union of $\lambda$ disjoint set-systems $\mathcal{I}_{\nu}$ ($\nu<\lambda$) in such a way that for any $A\in V'$, $\langle \mathcal{I}_\nu\restriction A: \nu <\lambda\rangle$ has simultaneous chromatic number $\geq \kappa$.
\end{definition}

\begin{corollary}\label{corollary: set-systems}
 $P([\kappa]^2, \aleph_0, \aleph_1, \kappa\cap \cof(\omega_1))$ holds for any cardinal $\kappa\geq \omega_2$.
\end{corollary}

\begin{proof}
Let $c: [\kappa]^2\to \omega$ be given by Theorem \ref{theorem: main}. For each $n\in \omega$, let $\mathcal{I}_n = c^{-1}(\{n\})$. For any $\gamma\in \kappa\cap \cof(\omega_1)$ and for any $f: \gamma\to \omega$, there is some $m\in \omega$ and $A\subseteq \gamma$ of order type $\omega_1$ that is stationary in $\gamma$ such that $f'' A =\{m\}$. By the property of $c$, for all $n\in \omega$, there are $\alpha<\beta\in A$ such that $c(\alpha,\beta)=n$, namely, $(\alpha,\beta)\in \mathcal{I}_n \restriction \gamma$. Hence, $\langle \mathcal{I}_n\restriction \gamma: n\in \omega\rangle$ has simultaneous chromatic number $\geq \aleph_1$, as desired.
\end{proof}

\section{Open questions}\label{section: questions}

\begin{question}
For an uncountable cardinal $\kappa$, is there a coloring $c: [\kappa]^2\to \omega_1$ such that  $c''[A]^2 = \omega_1$ for any \emph{closed} $A\subseteq \kappa$ of order type $\omega_1$?
\end{question}

The following weaker conclusion is also of interest. 

\begin{question}
For a regular uncountable cardinal $\kappa$, is there a coloring $c: [\kappa]^2\to \omega_1$ such that $|c''[A]^2| = \aleph_1$ for any \emph{closed} $A\subseteq \kappa$ of order type $\omega_1$?
\end{question}

\bibliographystyle{alpha}
\bibliography{bib}

\end{document}